\documentclass[11pt,reqno]{amsart}
\usepackage{amsmath,amssymb,amscd,mathtools}
\usepackage[centering,textwidth=6.5in,textheight=9in]{geometry}
\usepackage{color,xcolor}
\usepackage{tikz}

\newtheorem{thm}{Theorem}[section]

\newtheorem{lem}[thm]{Lemma}
\newtheorem{pro}[thm]{Proposition}
\newtheorem{coro}[thm]{Corollary}

\theoremstyle{definition}

\newtheorem{remark}[thm]{Remark}
\newtheorem{example}[thm]{Example}

\numberwithin{equation}{section}

\def\bQ{\mathbb{Q}}
\def\bR{\mathbb{R}}
\def\bT{\mathbb{T}}
\def\bZ{\mathbb{Z}}

\def\cA{\mathcal{A}}

\def\eps{\epsilon}

\def\tf{\tilde{f}}
\def\tg{\tilde{g}}
\def\tw{\tilde{w}}
\def\tx{\tilde{x}}

\def\tA{\tilde{A}}
\def\tD{\tilde{D}}
\def\tU{\tilde{U}}

\def\pa{\partial}

\def\ds{\displaystyle}

\begin{document}

\title{Twist interval for twist maps}

\author{Pengfei Zhang}
\address{Department of Mathematics,
University of Oklahoma, Norman, OK 73019}
\email{pengfei.zhang@ou.edu}



\begin{abstract}
The twist interval of a twist map on the annulus $A=\bT\times [0,1]$ 
has nonempty interior if $f$ preserves the area,
but could be degenerate for general twist maps.
In this note, we show that if a twist map $f$ is non-wandering,
then the twist interval of $f$ is non-degenerate.
Moreover, if there are two disjoint invariant curves of $f$,
then their rotation numbers must be different 
(no matter if they are rational or irrational).
\end{abstract}

\maketitle

\section{Introduction}\label{intro}

Let $f$ be an orientation-preserving homeomorphism on the closed annulus $A=\bT\times [0,1]$, 
$(x,y)\mapsto (x_1,y_1)$.
Suppose $f$ preserves the two boundaries of $A$: $y_1=0$ if $y=0$, and $y_1=1$ if $y=1$.
The restriction of $f$ to each boundary $\bT\times \{i\}$, denoted by $f_i$, is a circle homeomorphism, $i=0, 1$.
Let $\rho(f_i)$ be the rotation number of $f_i$.
More generally, one can define the rotation set $I_f$ of $f$ on the whole annulus $A$,
see Section \ref{prelim} for more details. 
The rotation set $I_f$ could be complicated for general annulus maps.

The map $f$ on $A$ is said to satisfy a (positive) {\it twist condition} if for each $x\in \bT$,
the map $y\mapsto x_1(x,y)$ is strictly increasing. We will call such $f$ a  twist map.
For example, the map $f(x,y)=(x+y,y)$ satisfies the twist condition.

It is easy to see that 
for a twist map $f$ on $A$,  the rotation set of $f$ satisfies $I_f \subset[\rho(f_0),\rho(f_1)]$.
In the following, $[\rho(f_0),\rho(f_1)]$ will be called the {\it twist interval} of the twist map $f$.
Note that it is possible that $\rho(f_0)=\rho(f_1)$ for a twist map,
and hence the twist interval can degenerate to a single point, 
a phenomena caused by the {\it mode-locking effect}, see \S \ref{exam} for examples.

We need an extra condition to guarantee the  non-degeneracy of twist intervals.
Recall that a map $f$ is called {\it non-wandering} if the non-wandering set of $f$
equals the whole space.
\begin{thm}\label{mainthm}
Let $f$ be a non-wandering twist map on $A$.
Then $\rho(f_0)<\rho(f_1)$.
\end{thm}

As an application of the above theorem, we have the following result.
\begin{coro}
Let $f$ be a non-wandering  twist map on $A$. 
Then the rotation numbers of any two disjoint invariant curves of $f$  are different.
\end{coro}

A condition slightly `weaker' than the twist condition is the so-called {\it  boundary twist condition}.
Recall that an orientation-preserving homeomorphism $f$ on $A$ is said to satisfy
the boundary twist condition if $\rho(f_0) <\rho(f_1)$. 
As Example \ref{deg} shows, some twist map does not necessarily satisfy the `weaker' boundary twist condition.
It follows from Theorem \ref{mainthm} that
\begin{coro}\label{btc}
A non-wandering twist map on $A$ always satisfies the boundary twist condition.
\end{coro}

\section{Preliminary}\label{prelim}

In this section we introduce some notations and results that will be used later.

\subsection{Non-wandering set}
Let $f$ be a homeomorphism on a compact topological space $X$. A point $x\in X$
is called {\it wandering} if there is an open neighborhood $U$ of $x$
such that $f^n U\cap U=\emptyset$ for each $n\ge 1$.
Let $\Omega(f)$ be the set of points that are not wandering,
which is called the non-wandering set of $f$. Then $f$ is said to be {\it non-wandering}
if $\Omega(f)=X$. A point $x\in X$ is said to be {\it recurrent} if $f^{n_i} x \to x$ for some $n_i \to \infty$. 
Note that if  $f:X\to X$ is non-wandering, then the set of recurrent points are dense in $X$
and $f^n:X\to X$ is non-wandering for any $n\ge 1$.

\subsection{Lifts to the universal cover}
Let $A=\bT\times [0,1]$ be the closed annulus,
$\tA=\bR\times [0, 1]$ be the universal cover of $A$, $\pi_1$ be the projection
from $A$ to $\bT$. We will use the same notation for the projection from $\tA$ to $\bR$.
Let $f: A\to A$ be an orientation-preserving homeomorphism on $A$. 
Then one can lift the map $f$ from $A$ to its universal cover $\tA$.
The lift is unique up to an integer shift $T_k:(x,y)=(x+k, y)$, where $k\in \bZ$.
Let $\tf$ be such a  lift of $f$ to $\tA$.
Let $f_i$ be the projection of the restriction $f$ on $\bT\times \{i\}$ to $\bT$.
That is, $f_i(x)=\pi_1(f(x,i))$, $i=0, 1$.
In the same way we define  the projection $\tf_i$  of the restriction $\tf$ on $\bR\times \{i\}$ to $\bT$, $i=0,1$.

\subsection{Rotation numbers of circle homeomorphisms}
Let $g$ be an orientation-preserving homeomorphism on $\bT$,
$\tg$ be a lift of $g$ from $\bT$ to $\bR$.
Poincar\'e proved that the limit $\ds \lim_{n\to\infty}\frac{\tg^n(\tx) -\tx}{n}$
exists and is independent of the choices of $\tx\in\bR$.
Denote the limit by $\rho(\tg)$, which will be called the rotation number of $g$.
A different choice of the lift $\tg$ of $g$ results in an integer shift of
the rotation number.

It follows from the definition of rotation numbers that if the lifts of
two circle homeomorphism $g_1$ and $g_2$
satisfy $\tg_1(\tx)\le \tg_2(\tx)$ for each $\tx\in\bR$, then  
$\rho(\tg_1)\le \rho(\tg_2)$. However, a stronger condition  $\tg_1(\tx)< \tg_2(\tx)$ for each $\tx\in\bR$
does {\it not} necessarily lead to the stronger result that $\rho(\tg_1)< \rho(\tg_2)$.
\begin{pro}\label{irrational}
Assume  $\tg_1(\tx)< \tg_2(\tx)$ for each $\tx\in\bR$. If $\rho(\tg_1)$ is irrational,
then  $\rho(\tg_1)< \rho(\tg_2)$.
\end{pro}
See  \cite[Chapter 1]{dMvS} or \cite[Proposition 11.1.9]{KH95} for proofs of this result.

\subsection{Rotation sets of annulus maps}
Next we define the rotation set of a map $f$ on the annulus $A$.
For  a general point $(x,y)\in A$, we lift it to some point $(\tx, y)\in \tA$,
and denote $(\tx_n,y_n)=\tf^n(\tx, y)$. Then we define the
lower and upper rotation numbers of $(x,y)$ under $f$ as
$\ds \rho^\ast(x,y,f):=\limsup_{n\to\infty}\frac{\tx_n- \tx}{n}$,
$\ds \rho_\ast(x,y,f):=\liminf_{n\to\infty}\frac{\tx_n- \tx}{n}$.
The two limits coincide for $\mu$-a.e. $x\in A$ for every $f$-invariant probability measure $\mu$. 
Denote the common value by $\rho(x,y,f)$.
More generally,  the rotation set of $f$ on $A$ is defined by
\begin{align}\label{defset}
I_f=\Big\{\rho\in \bR: \frac{\pi_1(\tf^{n_i}(\tx_i,y_i)) -\tx_i}{n_i}\to \rho
\text{ for some }(\tx_i,y_i)\in \tA, n_i\to\infty\Big\}.
\end{align}
Note that $I_f$ is always closed.  See \cite{MZ91} for  a detailed discussion 
of rotation sets.

\subsection{Birkhoff's theorem on invariant curves}
Let $f$ be a twist map on $A$. An {\it invariant curve} of $f$ is an invariant 
circle in $A$ that goes around the annulus (hence not null-homotopic in $A$).
\begin{pro}\label{Birkhoff}
Let $f$ be a twist map on $A$. Then there exists a constant $L(f)>0$
such that any invariant curve of $f$ is the graph of some Lipschitz 
continuous function whose Lipschitz constant  is bounded by $L(f)$.
\end{pro}
For a proof of Birkhoff's theorem, see \cite{Birk}, or \cite[Lemma 13.1.1]{KH95}.

\section{Twist interval of twist maps}\label{main}

Let $A=\bT\times [0,1]$ be the annulus,
$f: A\to A$ be an orientation-preserving homeomorphism
that satisfies the twist condition. In the following we will simply say that $f$ is a {\it twist map}. 
Let $f_i$ be the projection of the restriction of $f$ on the boundary $\bT\times\{i\}$,
and $\rho_i:=\rho(f_i)$ be the rotation number of $f_i$, $i,=0, 1$, via some lift $\tf$.
A different choice of the lift $\tf$ results in a shift of $\rho_0$ and $\rho_1$ by the same integer.
We make the following convention:

\noindent {\bf Convention.} We always pick the lift $\tf$ of $f$ that satisfies $\rho(\tf_0)\in[0,1)$.

\subsection{The rotation set of twist maps}
In the following we will call $[\rho_0,\rho_1]$ the twist interval
of $f$.
Let $I_f$ be the rotation set of $f$ on $A$. See Section \ref{prelim}
for the definition of these quantities.

\begin{lem}
Let $f: A\to A$ be a  twist map. 
Then the rotation set  of $f$ satisfies $I_f\subset  [\rho_0,\rho_1]$.
\end{lem}
\begin{proof}
Let $\rho\in I_f$. Then according to \eqref{defset},
 $\ds \rho=\lim_{i\to\infty}\frac{\pi_1(\tf^{n_i}(\tx_i,y_i)) -\tx_i}{n_i}\in I_f$
for some  $(\tx_i,y_i)\in \tA$ and $n_i \to \infty$.
Let us fix the index $i$ for now. 

Let $(\tx_{i,n},y_{i,n})=\tf^n(\tx_i, y_i)$ be the $n$-th iterate of $(\tx_i, y_i)$,
and $(\tx_{i,n}',0)=\tf^n(\tx_i,0)$ be a comparison orbit.
We claim that $\tx_{i,n} \ge \tx_{i,n}'$ for each $n\ge 1$.
\begin{proof}[Proof of the claim]
By the twist condition, we have $\tx_{i,1} \ge \tx_{i,1}'$. 
Assume $\tx_{i,k} \ge \tx_{i,k}'$ for each $1\le k \le n$.
Then for $k=n+1$, we have 
\begin{align*}
\tx_{i,n+1}=\pi_1(\tf(\tx_{i,n},y_{i,n})) \ge \pi_1(\tf(\tx_{i,n},0))
=\tf_0(\tx_{i,n}) \ge \tf_0(\tx_{i,n}')=\tx_{i,n+1}',
\end{align*}
since $\tf_0$ preserves the order of the points.
Therefore, $\tx_{i,n} \ge \tx_{i,n}'$ for each $n\ge 1$.
\end{proof}
It follows from the above claim that 
$\ds \frac{\pi_1(\tf^{n}(\tx_i,y_i)) -\tx_i}{n} \ge \frac{\pi_1(\tf^{n}(\tx_i,0)) -\tx_i}{n}$
for any $n\ge 1$. Setting $n=n_i$ and then letting $i\to\infty$, we see that
$\rho \ge \rho_0$. 
In the same way we have $\rho \le \rho_1$. This holds for any $\rho\in I_f$.
Therefore, $I_f\subset [\rho_0,\rho_1]$.
\end{proof}
Without some extra assumption of $f$, it is possible that 
the rotation set $I_f \subsetneq [\rho_0,\rho_1]$.
See \S \ref{exam} for examples of twist maps with $I_f= \{\rho_0,\rho_1\}$.

By our twist condition, we know that $\tf_0(\tx)<\tf_1(\tx)$ for any $\tx\in \bR$.
It follows from the definition of rotation numbers that $\rho_0\le \rho_1$.
This inequality may not necessarily be a strict one.
See Example \ref{deg}, where a twist map has a degenerate twist interval.

\subsection{Non-wandering twist maps}
In this subsection we consider the case when $f$ is non-wandering.
Recall that if $f$ is non-wandering, so is $f^n$ for each $n\ge 1$.
\begin{thm}\label{rho-interval}
Let $f: A\to A$ be a  twist map. 
If $f$ is non-wandering, 
then $\rho_0< \rho_1$.
\end{thm}
\begin{proof}
We will assume $\rho_0= \rho_1$ and derive a contradiction from it.

\noindent{\bf Case 1.} 
$\rho_0$ is irrational. The twist condition implies 
$f_0(x)<f_1(x)$ for any $x\in \bT$. 
Then Proposition \ref{irrational} states $\rho_0 <\rho_1$, a contradiction.

\noindent{\bf Case 2.}  $\rho_0=p/q$ is a rational number. We start with the special case
$\rho_0=0$ and then extend our proof to the general case.

\noindent{\bf Case 2a.} $\rho_0=0$. It means $f_i$ admits some fixed point for each $i=0,1$.
By our choice of the lift $\tf$, we see that $\tf_i$ also admits some fixed point, $i=0,1$.
Let $\tx_0\in \bR$ be a fixed point of $\tf_0$
and $\tx_1\in(\tx_0, \tx_0+1]$ be the corresponding fixed point of $\tf_1$.
Then $(\tx_i,i)\in \tA$, $i=0,1$ are two fixed points of $\tf$. 
By the twisting condition, $\tf_{1}(\tx_0) > \tf_0(\tx_0)$. Moreover, since $\tf_i$ is orientation-preserving, 
we see that $\tf_1^n(\tx_0)\in(\tx_0, \tx_1)\subset \bR$ for any $n\ge 1$.

Consider the vertical segment $L_{0}=\{\tx_0\}\times [0,1]$, and let $L_1:=\tf(L_0)$ under $\tf$ be its image.
Applying the twist condition again, we see that $L_0$ and $L_1$ intersect only at $(\tx_0,0)$,
and they bound a triangular domain $\tD_0\subset \tA$.
Note that 
$\tf^{n+1}L_0\cap \tf^nL_0=\tf^n(\tf(L_0)\cap L_0)=\tf^n\{(\tx_0,0)\}=\{(\tx_0,0)\}$ for any $n\ge 1$.
The orientation-preserving assumption of $f$ implies 
that $\tf^{n+1}L_0$ lies on the right hand side of $\tf^nL_0$,
and the domain bounded by them is exactly $\tf^n(\tD_0)$.
So $\tf^n(\tD_0)$, $n\ge 0$, are mutually disjoint, and lie on the right side of $L_0$.
Therefore, $\tD_0$ is a wandering domain with respect to $\tf$ on $\tA$.

Let $D_0$ be the projection of $\tD_0$ on $A$. Since $f$ is non-wandering,
$f^n(D_0)\cap D_0\neq\emptyset$ for some $n\ge 1$.
Lifting to $\tA$, we see that  $\tf^n L_0$ has to reach to the right side to
the shifted segment $L_0+(1,0)=\{\tx_0+1\}\times [0,1]$.
Then $\tf^n L_0$ crosses $L_0+(1,0)$ at least twice 
since the two endpoints of $\tf^n L_0$ are kept on the left side of the 
the vertical segment $\{\tx_1\}\times [0,1]$, and $\tx_1\le \tx_0+1$. 
Projecting this structure from $\tA$ to $A$, we see that
there exists a topological horseshoe $\Lambda\subset A$ that is invariant under $f^n$.
Let $w\in \Lambda$ be a point fixed by $f^n$.
Then the lift $\tw\in \tA$ of $w$ satisfies $\tf^n (\tw)= \tw+(1,0)$, and hence $\rho(\tw, \tf)=1/n$. 
It contradicts the hypothesis that $I_f=\{0\}$.

\noindent{\bf Case 2b.} Now we deal with the general case that $\rho_0=p/q$. 
Our argument is similar to Case 2a, with not one, but $q$ moving screens.
Let $x_0$ be a periodic point $f_0$, $\tx_0$ be a lift of $x_0$. Then 
$\tf^{nq}_0(\tx_0)=\tx_0+np$ for any integer $n$.
Let  $\tx_1\in(\tx_0, \tx_0+1]$ be the corresponding periodic point of $\tf_1$.
We see that $\tf_1^{qn}(\tx_0)\in(\tx_0+np, \tx_1+np)\subset \bR$ for any $n\ge 1$.
On $\tA$, the lift $\tf$  satisfies $\tf^q(\tx_0,0)=(\tx_0+p,0)$.
Let $L_k=\{\tf_0^k\tx_0\}\times [0,1]$ for each $k \ge 0$.
Then $\tf(L_k)$ lies on the right side of $L_{k+1}$ by the twist condition.
Let  $\tD_k$ be the domain bounded by $\tf(L_k)$ and $L_{k+1}$, $0\le k\le q-1$.
Then the domain bounded by $\tf^q(L_0)$ and $L_{q}=\{\tx_0+p\}\times[0,1]$ 
is $\tU:=\tD_{q-1}\cup \tf(\tD_{q-2})\cup\cdots \cup \tf^{q-1}(\tD_0)$.

Let $\tg(\tx,y)=\tf^q(\tx,y)-(p,0)$. Then the domain bounded by $L_0$ and $\tg(L_0)$
is exactly $\tU-(p,0)$. Therefore, $\tU-(p,q)$ is a wandering domain with respect to $\tg$.
The projection $g$ of $\tg$ satisfies $g(x,y)=f^q(x,y)$, and hence is non-wandering.
Using the same argument as in Case 2a, we see that there exists a $g$-periodic point 
$w\in A$ with $\rho(\tw,\tg)>0$.
On the other hand, $I_g=q\cdot I_f -p=\{0\}$, which leads to a contradiction.
This completes the proof.
\end{proof}

We consider a special case when $f$ preserves a fully supported measure.
Let $\mu$ be a probability measure on $A$ that is fully supported. 
That is, $\mu(U)>0$ for any nonempty open set $U\subset A$.
\begin{coro}\label{full}
If a  twist map $f: A\to A$ 
preserves a fully supported measure, 
then $\rho_0< \rho_1$.
\end{coro}

\vskip.1in

\begin{remark}
Let $f$ be a non-wandering twist map. Then Theorem \ref{rho-interval} shows 
$\rho_0<\rho_1$.
Combining with Franks' generalized Poincar\'e--Birkhoff Theorem
\cite{Fra88}, we see that $I_f\supseteq \bQ\cap [\rho_0, \rho_1]$.
Therefore $I_f=[\rho_0, \rho_1]$ since $I_f$ is closed.
Note that Franks proved much stronger results  in \cite{Fra88} than what we need here.
\end{remark}

The following is a direct corollary of Theorem \ref{rho-interval}.
\begin{coro}\label{disjoint}
Let $f$ be a non-wandering twist map on $A$.
Then any two disjoint invariant curves of $f$ have different rotation numbers.
\end{coro}
\begin{proof}
Let $C_1$ and $C_2$ be two invariant curves of $f$ that are disjoint. 
Proposition \ref{Birkhoff} states that each $C_i$ is the graph of some continuous (in fact Lipschitz) function 
$\phi_i:\bT\to [0,1]$. Since $C_1$ and $C_2$ are disjoint,
we assume $\phi_1(x)<\phi_2(x)$ for any $x\in \bT$.
Then the region $A'$ between $C_1$ and $C_2$ is a smaller annulus
and the restriction $f|_{A'}$ is a twist map that is also non-wandering. 
Then we can apply Theorem \ref{rho-interval} 
and conclude that the two rotation numbers of $f$ on $C_1$ and $C_2$
are different.
\end{proof}

\vskip.1in

\begin{remark}
Note that for any twist map, there is {\it at most one} (disjoint or not)
invariant curve with rotation number $\rho$ {\it if $\rho\in I_f$ is irrational}.
See  \cite{Mat85} or \cite[Theorem 13.2.9]{KH95}.
The phase portrait of an elliptic billiards (see \cite[Page 12]{CM06})
indicates that there can be more than one (non-disjoint)
invariant curves of the same rational rotation number (even when the map
preserves a smooth measure).
If $f$ admits more than one invariant curves with the same rotation number $\rho$,
then $\rho$ is rational, and 
all these curves intersect along some common Birkhoff periodic orbits of $f$.
\end{remark}

\begin{proof}
it follows from \cite[Theorem 13.2.9]{KH95} that $\rho$ must be rational, say $p/q$.
Let $\{C_{\alpha}:\alpha\in \cA\}$ be the collection of invariant curves with the  rotation number $\rho$.
It follows from Proposition \ref{Birkhoff} that for each $\alpha\in \cA$,  
$C_\alpha=\{(x,\phi_\alpha(x)): x\in \bT\}$
for some Lipschitz functions $\phi_\alpha: \bT\to [0,1]$ with a uniform Lipschitz constant.
Let $\psi_1(x)=\inf_\alpha\{\phi_\alpha(x)\}$ and $\psi_2(x)=\sup_\alpha\{\phi_\alpha(x)\}$. 
Then $\psi_1$ and $\psi_2$ are two Lipschitz functions, whose graphs $\gamma_1$
and $\gamma_2$
are invariant curves of $f$ with rotation number $\rho$.

It follows from Corollary  \ref{disjoint} that the intersection $E:=\gamma_1 \cap \gamma_2$
is nonempty, 
which is also closed and $f$-invariant. Let $X=\pi_1(E)=\{x\in\bT: \psi_1(x)=\psi_2(x)\}$,
and enumerate the component complement $\bT\backslash X$, say $I_n$, $n\ge 1$.

Then for each $n\ge 1$, the two invariant curves $\gamma_1$
and $\gamma_2$ bound an open disk $D_n$ over $I_n$.
Since both $\gamma_1$ and $\gamma_2$ are invariant, these disks are permuted by $f$.
The non-wandering property of $f$ implies all of the disks are periodically permuted,
and the corresponding points in the intersection $E$ must be periodic.
These periodic points are of Birkhoff type $p/q$, since they lie on an invariant curve of $f$
of rotation number $p/q$.
\end{proof}

\section{Some examples of twist maps}\label{exam}

In this section we give some example to illustrate the different situations of twist maps.
We start with a standard one, and then give some variations of it.
These classes of examples show that some extra assumption is needed 
to get non-degenerate twist intervals and rotation sets with nonempty interior.

\begin{example}\label{circle}
Let $f(x,y)=(x+\phi(y), y)$, $(x,y)\in A$, where $\phi:[0,1]\to \bR$ is a continuous
and increasing function. Then $f$ is a twist map that preserves the Lebesgue measure on $A$,
and the rotation set $I_f=[\phi(0),\phi(1)]$.
\end{example}

Our first class of variations of the standard twist map  is
\begin{example}\label{float}
Let $f(x,y)=(x+\phi(y), \psi(y))$, $(x,y)\in A$, where $\phi:[0,1]\to \bR$ is increasing, 
and $\psi$ is a homeomorphism on $[0,1]$ that fixes the two endpoints.  
Let $\text{Fix}(\psi)=\{y\in [0,1]: \psi(y)=y\}$ be the set of points fixed by $\psi$.
Then $f$ is a twist map whose nonwondering set is $\Omega(f)=\bT\times \text{Fix}(\psi)$.
The rotation set $I_f=\{\phi(y): y\in \text{Fix}(\psi)\}$ can be any closed subset of $[\phi(0), \phi(1)]$.
\end{example}

\vskip.1in

To introduce the second variations,
we briefly recall the Mode Locking phenomena in circle dynamics. 
See \cite[Section 1.4]{dMvS} for more details.
Let $g_0:\bT\to \bT$ be a circle homeomorphism with a periodic point $x_0\in \bT$ of period $p/q$.
Then $\rho(g_0)=p/q$.
Assume the graph of the $q$-th iterate $g_0^q$ crosses the diagonal at $x_0$.
Then for any circle map $g$ that is close to $g_0$, 
the graph of $g^q$ also crosses the diagonal, which implies that the rotation number $\rho(g)$ of $g$
is locked at $p/q$. 

\begin{example}\label{deg}
Let $f_0$ be a  circle homeomorphism with locked mode $p/q$.
Consider the one-parameter family $\{f_t: t\in\bR\}$ of maps, where $f_t: x\in \bT \mapsto f_0(x)+t$.
Then there exists $\eps_0>0$ 
such that $\rho(f_t) =p/q$ for any $t\in[0, \eps_0]$.
Consider the map $f$ on $\bT\times [0,\eps_0]$, define by $f(x,t):=(f_t(x),t)$.
It is easy to see that $f$ satisfies the twist condition, $\rho_0=\rho_1=p/q$ and $I_f=\{p/q\}$.
\end{example}

One might wonder 
what one can say when $\Omega(f)$ has nonempty interior.
To construct our next examples, we first recall the phase portrait of
the billiard map inside an ellipse. See \cite[\S 1.4]{CM06} for more details.

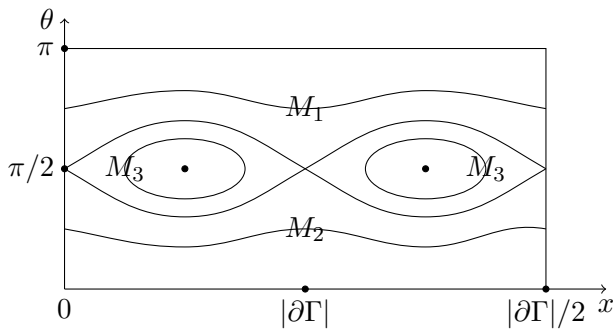
\begin{figure}[h]
\begin{tikzpicture}[scale=0.8]
\draw[->] (0,0) -- (9,0) node[pos=0, align=left, below]{$0$}  node[pos=1, align=left,   below]{$x$};
\draw[->] (0,0) -- (0,4.5) node[left]{$\theta$};
\draw[fill] (4,0) circle (.3ex) node[below]{$|\pa\Gamma|$};
\draw[fill] (8,0) circle (.3ex) node[below]{$|\pa\Gamma|/2$};
\draw[fill] (0,2) circle (.3ex) node[left]{$\pi/2$};
\draw[fill] (0,4) circle (.3ex) node[left]{$\pi$};
\draw (0,4) -- (8,4) -- (8,0);
\draw (0,2)  to[out=30, in=180] (2,2.8)  to [out=0, in=150] (4,2) to [out=30, in=180] (6,2.8)  to [out=0, in=150] (8,2);  
\draw (0,2)  to[out=-30, in=180] (2,1.2)  to [out=0, in=210] (4,2) to [out=-30, in=180] (6,1.2)  to [out=0, in=210] (8,2);  
\draw (2,2) ellipse (1 and .5);
\draw (6,2) ellipse (1 and .5);
\draw[fill] (2,2) circle (.3ex);
\draw[fill] (6,2) circle (.3ex);
\draw (4,3)   node{$M_1$};
\draw (4,1)  node{$M_2$};
\draw (1,2)   node{$M_3$};
\draw (7,2)  node{$M_3$};
\draw (0,3)  to[out=10, in=180] (2,3.3)  to [out=0, in=180] (4,3) to[out=0, in=180] (6,3.3)  to [out=0, in=170] (8,3); 
\draw (0,1)  to[out=-10, in=180] (2,.7)  to [out=0, in=180] (4,1) to[out=0, in=180] (6,.7)  to [out=0, in=170] (8,1); 
\end{tikzpicture}
\caption{Phase portrait of the billiard dynamics inside an ellipse.}
\label{phase}
\end{figure}

Let $\Gamma$ be the ellipse $\frac{x^2}{a^2}+\frac{y^2}{b^2}=1$ with $a>b>0$.
Consider the billiards system inside $\Gamma$.
Then the phase space of the billiard map $F$ is given by $M:=(\bR/|\Gamma|)\times [0,\pi]$,
where $x\in \bR/|\Gamma|$ is the arc-length parameter of $\Gamma$,
and $\theta\in[0,\pi]$ measures the angle from the tangent vector $\Gamma'(x)$
to the velocity vector of the orbit right after the impact on $\Gamma$.
Note that $\frac{dx_1}{d\theta}=\frac{\tau(x,x_1)}{\sin\theta}$,
where $\tau(x,x_1)$ is the Euclidean distance from $\Gamma(x)$ to $\Gamma(x)$,
see  \cite[\S 2.11]{CM06}.
There exists a constant $c=c(\Gamma)>0$ such that $\frac{dx_1}{d\theta} \ge c$.
So $F$ satisfies the twist condition, and $I_F=[0,1]$.

\begin{example}
The phase space $M$ is divided into three parts. Let $M_1$ be the upper part,
$M_2$ be the lower part, and $M_3$ be the center part.
We make a smooth perturbation $G$ of $F$ on the interior of $M_1$
that pushes the invariant curves in $M_1$ upward, 
and on the interior of $M_2$ that pushes the invariant curves in $M_2$ downward,
while keeps $F$ unchanged on $M_3$.
Clearly $G$ satisfies the boundary twist condition and $\Omega(G)\supset M_3$.
Moreover, $G$ still satisfies the twist condition (as long as the perturbation is $C^1$-small)
while $I_G=\{0, 1/2, 1\}$.
\end{example}

For our last example, we insert the dynamics of
the elliptic billiards on the eye-shape domain $M_3$ into twist map with degenerate 
twist interval.
Let $f_0: \bT\to \bT$ be a diffeomorphism with rotation number $\rho_0=1/2$,
such that $f_0(x)=x+1/2$ for $1/6\le x\le 1/3$ and $f_0(x)=x-1/2$ for $2/3\le x\le 5/6$,
and all other points are wandering. See Part (a) of Fig.~\ref{example5}.
There exists $\eps_0>0$ such that
$f_\eps(x):=f_0(x)+\eps$ has a unique periodic orbit of period $2$ whenever 
$0<|\eps|\le\eps_0$.
Consider the induced twist map $f$ on A:=$\bT\times[-\eps_0,\eps_0]$, $f(x,y)=(f_0(x)+y,y)$.
Then $\frac{\pa x_1}{\pa y}=1>0$ and $\rho(f_{-\eps_0})=\rho(f_{\eps_0})=1/2$. 
See Part (b) of Fig.~\ref{example5} for the non-wandering set of $f$.

\begin{example}
Now we make a (piecewise) $C^1$ small perturbation $g$ of $f$ over the two cylinders. 
$[1/6,1/3]\times [-\eps_0,\eps_0]$ and $I_2=[2/3,5/6]\times [-\eps_0,\eps_0]$.
We first cut $A$ along the two flat segments
$I_1=[1/6,1/3]\times\{0\}$ and $I_2=[2/3,5/6]\times\{0\}$, push the upper copy to the right,
and the lower copy to the right
and then paste the restriction of the dynamics of the elliptic billiards $F$ on $M_3$. 
This perturbation resembles the change of the billiard map 
when one deforms the billiard table from a unit disk to an elliptic domain.
Note that $g$ is $C^1$-close to $f$ outside the two eyes, and equals to $F$ on the two eyes.
So $g$ is a twist map.
It is easy to see $\Omega(f)$ has nonempty interior, while $\rho_0=\rho_1=1/2$.
\end{example}

\begin{figure}[h]
\begin{minipage}{0.4\linewidth}
\begin{tikzpicture}[scale=0.7]
\draw[->] (-1,0) -- (7,0) node[pos=.125, align=left, below]{$0$} 
node[pos=.875, below]{$1$} node[pos=1, align=left,   below]{$x$};
\draw[->] (0,0) -- (0,6.5) node[pos=.92, left]{$1$} node[left]{$y$};
\draw[dotted] (0,6) -- (6,6) -- (6,0) (0,0) -- (6,6) (0,3) -- (3,6)  -- (3,0) -- (6,3) -- (0,3);
\draw[fill] (1,0) circle (.2ex);
\draw[fill] (2,0) circle (.2ex);
\draw[fill] (4,0) circle (.2ex);
\draw[fill] (5,0) circle (.2ex);
\draw  (0,3.3)  to [out=40, in=225]  (1, 4) -- (2,5) to [out=40, in=225] (3.3,6);
\draw  (3.3,0) to [out=45, in= 225]  (4, 1) -- (5,2) to [out=40, in=225] (6, 3.3);
\node at (3,-.5){(a)};
\end{tikzpicture}
\end{minipage}
\begin{minipage}{0.4\linewidth}
\begin{tikzpicture}[scale=0.7]
\draw[dotted] (0,4) -- (6,4);
\draw[dotted] (0,4) -- (0,6) node[pos=0, left]{$-\eps_0$}
 node[pos=0.5, left]{$0$}  node[pos=1, left]{$\eps_0$};
\draw[dotted]  (0,5) -- (6,5)  (0,6) -- (6,6) (6,4) -- (6,6);
\draw  (1, 5) -- (2,5) (4, 5) -- (5,5);
\draw[color=blue] (.6,4)  to [out=45, in=250]  (1, 5)  (2,5) to [out=70, in=225] (2.4,6);
\draw[color=red] (2.6,6)  to [out=-45, in=110]  (3, 5) to [out=-70, in=135] (3.4,4);
\draw[color=blue] (3.6,4)  to [out=45, in=250]  (4, 5)  (5,5) to [out=70, in=225] (5.4,6);
\draw[color=red] (5.6,6)  to [out=-45, in=110]  (6, 5)   (0,5) to [out=-70, in=135] (.4,4);
\node at (3,3.5){(b)};
\draw[dotted] (0,0) -- (6,0);
\draw[dotted] (0,0) -- (0,2) (6,0) -- (6,2);
\draw[dotted]  (0,1) -- (6,1)  (0,2) -- (6,2);
\draw[color=blue] (.6,0)  to [out=45, in=250]  (1, 1)  (2,1) to [out=70, in=225] (2.4,2);
\draw[color=red] (2.6,2)  to [out=-45, in=110]  (3, 1) to [out=-70, in=135] (3.4,0);
\draw[color=blue] (3.6,0)  to [out=45, in=250]  (4, 1)  (5,1) to [out=70, in=225] (5.4,2);
\draw[color=red] (5.6,2)  to [out=-45, in=110]  (6, 1)   (0,1) to [out=-70, in=135] (.4,0);
\filldraw[fill=gray, draw=black] (1.5,1) ellipse (.5 and .15);
\filldraw[fill=gray, draw=black] (4.5,1) ellipse (.5 and .15);
\node at (3,-.5){(c)};
\node at (1.5,1.15){$\scriptscriptstyle >$};
\node at (4.5,1.15){$\scriptscriptstyle >$};
\node at (1.5,0.85){$\scriptscriptstyle <$};
\node at (4.5,0.85){$\scriptscriptstyle <$};
\end{tikzpicture}
\end{minipage}
\caption{Construction of a twist map by Cut-and-Paste: (a) graph of the function $f_0$; 
(b) the non-wandering set of the twist map $f$; (c) the perturbation $g$ of $f$.
Blue color indicates the periodic orbit is attracting, while red color indicates repelling.}
\label{example5}
\end{figure}
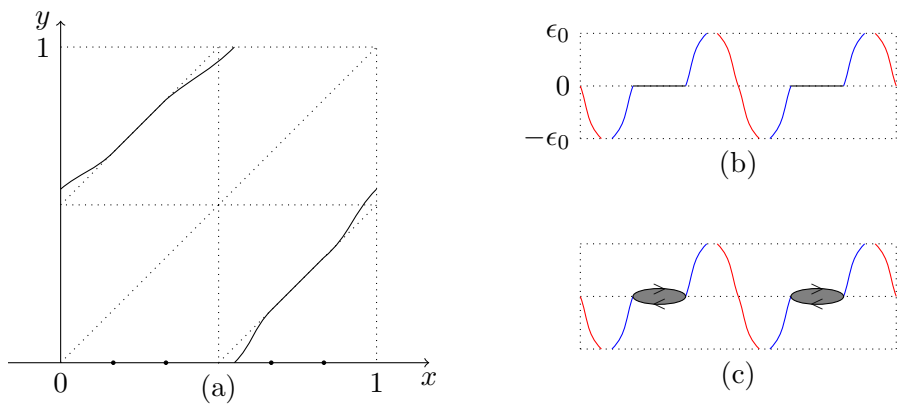

\section*{Acknowledgments}
The author is very grateful to Zhihong Xia for useful discussions.

\end{document}